\documentclass[12pt]{amsart}
\usepackage{amsfonts}
\usepackage{amsmath}
\usepackage{amssymb}
\usepackage{amsthm}
\usepackage{amsrefs}
\usepackage{fullpage}
\usepackage{color}
\usepackage[doublespacing]{setspace}

\def\draft{0}

\newcommand{\emptysequence}{e}

\theoremstyle{plain}
\newtheorem{theorem}{Theorem}
\newtheorem{proposition}{Proposition}
\newtheorem{lemma}{Lemma}

\newtheorem{corollary}{Corollary}

\theoremstyle{definition}
\newtheorem{definition}{Definition}
\newtheorem{example}{Example}

\theoremstyle{remark}

\newtheorem*{claim*}{Claim}

\theoremstyle{plain}
\newtheorem*{martins-theorem}{Martin's Theorem}
\newtheorem*{kuhns-theorem}{Kuhn's Theorem}
\newtheorem*{definetti-theorem}{De-Finetti's Theorem}
\ifnum\draft=1
\newcommand{\eran}[1]{\begin{framed}{\noindent \scriptsize{\bf  Eran's Note:}  #1}\end{framed}}
\else
\newcommand{\eran}[1]{}
\fi

\begin{document}
\title{A Prequential Test for Exchangeable Theories}
\author{Alvaro Sandroni}
\email{sandroni@kellogg.northwestern.edu}
\address{Kellogg School of Management, Northwestern University}
\author{Eran Shmaya}
\email{e-shmaya@kellogg.northwestern.edu}
\address{Kellogg School of Management, Northwestern University}

\begin{abstract}
We construct a prequential test of probabilistic forecasts that does not
reject correct forecasts when the data-generating processes is exchangeable
and is not manipulable by a false forecaster.
\end{abstract}

\maketitle

\section{Introduction}

A potential expert, named Bob, claims that he is informed about which
stochastic process will generate the future data. A tester, named Alice,
does not know whether Bob does, in fact, know anything relevant about the
process (i.e., something Alice does not know). However, Bob will repeatedly
announce a forecast of next period's outcomes. A basic question is whether
Alice can eventually determine whether Bob is informed. A growing literature
answers this question in the negative (see Olszewski (2011) for a literature
review).

A prequential test takes, as input, Bob's forecasts and realized outcomes
and returns, as output, a reject or pass verdict. Sandroni (2003) shows that
if a perfectly informed expert (i.e., predicts according to the data
generating process) is likely to pass Alice's test then a completely
uninformed expert (i.e., knows nothing about the data generating process)
can strategically concoct forecasts that are likely to pass the test,
regardless of the realized outcomes. Thus, Alice's test cannot screen
informed and uninformed, but strategic, experts. Olszewski and Sandroni
(2008) extended this result to finite, possibly unbounded horizon.

%Olszewski and
%Sandroni (2008) show that if a perfectly informed expert (i.e., predicts
%according to the data generating process) is likely to pass Alice's test
%then a completely uninformed expert (i.e., knows nothing about the data
%generating process) is equally likely to pass the same test, provided that
%he forecast strategically to manipulate the test. Thus, Alice's test cannot
%screen informed and uninformed, but strategic, experts.

There are no clear benefits of running a test if it is known, at the outset,
that it will return a pass verdict. So, the results in this literature cast
doubt on the use of empirical tests when experts are potentially strategic.
However, two key assumptions underlying the impossibility results should be
examined. First, at the outset, Alice knows nothing about how the data will
be produced. Second, the data available to her is finite.%\footnote{%
%In addition, it is also implicitly assumed that the experts know how they
%will be tested. The successful forecasting schemes of the uninformed expert
%depends upon the test. We maintain this assumption in this paper.}

The first condition (Alice's complete ignorance) imposes no restrictions on
the types of stochastic processes that may generate the data. So, the
results hold for tests that are unlikely to reject \textit{any} data
generating process. In contrast, consider the case in which Bob and Alice
assume that the data will be produced within a class of data generating
processes, such as exchangeable processes. Then Alice may look for tests
that do not reject any exchangeable process and she may require Bob to
forecast consistently with some exchangeable process. This allows for
empirical testing that relies on a priory assumptions on how the data might
be produced.

The second condition (finite data sets) is typically justified based on the
practical observation that the data cannot feasibly be infinite. However,
the impossibility results in this literature are conceptual: they show the
existence of forecasting schemes that can manipulate Alice's test, but they
do not determine how to construct such manipulating schemes in practice. The
assumption of finite data sets may be significant because the forecasting
schemes that manipulate the test may involve complex processes that require
large, potentially infinite, streams of data to be effectively tested.

It is tempting to speculate that Alice's ignorance, combined with finite
data sets, is central to the impossibility of screening informed and
uninformed experts. An ignorant, but strategic, expert may be able to
exploit the tester's ignorance, combined with her limited data sets, to
manipulate her test. However, the impossibility theorems still hold even if
any one of these two conditions is disposed:

Assume that it is known, to both Alice and Bob, that the data generating
process is an exchangeable processes. Alice does not know which exchangeable
process runs the data and Bob claims that he knows the exact process (before
any data is observed). Then, Alice has, a priory, a significant
understanding of how the data will be produced. She knows that the data is
produced by some exchangeable process and her empirical test may rely on
this condition. Yet, with finite data %, but perhaps unbounded data sets, 
the impossibility results still hold: Any prequential test that is likely to
pass an informed expert can be manipulated by an uninformed expert, and so
is equally likely to pass an uninformed expert even if he is restricted to
produce forecasts based on an exchangeable process (Proposition 1 below).%
% (see Olszewski and Sandroni (2007a) for a proof based on Fan's Minmax Theorem).

Now consider the case of infinite data sets, but with no restrictions on the
types of stochastic processes that may generate the data. %Using Martin's
%Theorem about determinacy of Blackwell Games (Martin (1998)), 
Shmaya (2008) showed that the impossibility results still hold: Any
prequential test that is likely to pass an informed expert is also likely to
pass an uninformed, but strategic, expert (Proposition 2 below). 
%This follows even if the tester has an infinite sequence of next period's predictions and realizations at her disposal. 

Thus, neither a restriction to exchangeable processes, nor infinite
sequences of forecasts and outcomes, can, by itself, rule out the
impossibility theorems. The question addressed in this paper is whether the
impossibility theorem still hold if it is assumed \textit{simultaneously}
that the data generating process is exchangeable and that the tester can
obtain large, perhaps infinite, sequences of forecasts and outcomes. Under
these conditions, we show that the impossibility theorems do not hold. We
construct a prequential test that is likely to pass an informed expert and
cannot be manipulated by an uninformed expert; no matter which forecasting
scheme he uses, he cannot be assured that he is likely to pass the test.

Here is the basic idea of our test: By De-Finetti's Theorem, every
exchangeable process has a representation as a sequence of i.i.d.\ coin
tosses with an unknown parameter $q$. By the law of large numbers, the
infinite realization pins down the parameter $q$. In addition, we show that
after observing Bob's forecasts along the realization Alice can deduce Bob's
prior distribution on the parameter. So, the sequential framework reduces to
a one-shot framework in which Bob delivers a prior that (he claims)
generates a parameter $q$ that is observed by Alice. The existence of
non-manipulable test can now be obtained by adapting a result in Olszewski
and Sandroni (2009b) to this one-shot framework.

Our argument reveals a non-intuitive distinction between the infinite and
finite horizon setups. With infinite streams of forecasts and outcomes, it
is possible to determine $q$ and, hence, to screen informed and uninformed
experts. With finite (perhaps arbitrarily long) streams of forecasts and
outcomes, it is possible to make arbitrarily accurate statistical inferences
about $q$ and yet this is insufficient to screen informed and uninformed
experts.

% In the case of unrestricted processes, the same impossibility theorem holds
% with finite and infinite sequences of forecasts and outcomes. However, if
% processes are restricted to be exchangeable then the results with a finite %,perhaps unbounded, 
% sequences of forecasts and outcomes, are fundamentally
% different from the results with infinite sequences of forecasts and
% outcomes. In the latter case, the impossibility theorems do not hold. In the
% former case, they do.

\section{Setup}

\label{setup} Let $S$ be a finite set of \emph{outcomes}. Every period $%
n=0,1,\dots$ an outcome is realized. Let $\Omega=S^\mathbb{N}$ be the set of 
\emph{realizations} equipped with the product topology. A \emph{theory} is a
probability measure over $\Omega$, representing the distribution of some
stochastic process. The set $\Delta(\Omega)$ of theories is a convex and
compact metrizable subspace of the topological vector space of all signed
measures equipped with the weak-$^\ast$ topology. A \emph{paradigm} is a
set of theories.

%For every realization $x\in S^\omega$ and every $n\in\omega$ we let $x|_n$ be the initial segment of $x$ of length $n$. 
Let $S^{<\mathbb{N}}=\bigcup_{n\ge 0}S^n$ be the set of \emph{partial
realizations} including the empty sequence $\emptysequence$. For every partial realization $\sigma=(s_0,\dots,s_{n-1})$
let $N(\sigma)\subseteq S^\mathbb{N}$ be the set of realizations $x\in S^\mathbb{N}
$ that extends $\sigma$, i.e. such that $\sigma$ is the initial segment of $x
$ of length $n$. In particular $N(\emptysequence)=S^\mathbb{N}$.
%$x|_n=\sigma$. %For every partial realization $\sigma=(s_0,\dots,s_{n-1})\in S^{<\bbn}$ and every outcome $s\in S$ let $\sigma\concat s=(s_0,\dots,s_{n-1},s)$.

\begin{definition}
\label{forecast}Let $\mu$ be a theory and $x\in S^\mathbb{N}=(s_0,s_1,\dots)$
a realization. The \emph{forecast of $\mu$ over $x$ at period $n$} is the
element $p_n\in \Delta(S)$ that is given by 
\begin{equation}  \label{p-n-s}
p_n[s]=\mu(s|s_0,\dots,s_{n-1})=\frac{\mu\left(N\left(s_0,\dots,s_{n-1},s%
\right)\right)}{\mu\left(N\left(s_0,\dots,s_{n-1}\right)\right)}
\end{equation}
for every outcome $s\in S$. That is, the forecast at period $n$ is
the forecast made at period $n-1$ (based on all information available at
period $n-1$) on the outcome at period $n$. Here and later, when $n=0$ the sequence $(s_0,\dots,s_{n-1})$ refers to the empty sequence $\emptysequence$. The forecast $p_n$ is undefined
if $\mu\left(N\left(s_0,\dots,s_{n-1}\right)\right)=0$. %When we
%want to emphasize the dependence on $\mu$ and $x$ we also denote $%
%\mu(s|s_0,\dots,s_{n-1})$ for $p_n[s]$.
\end{definition}

A \emph{prequential test} is given by a Borel function $T:(\Delta(S)\times
S)^\mathbb{N}\rightarrow\{\text{FAIL},\text{PASS}\}$: Bob passes the test if 
$T(p_0,s_0,p_1,s_1,\dots)=\text{PASS}$, where, for every $k\in\mathbb{N}$, $%
p_k$ is the forecast on period $k$ according to Bob's theory and $s_k$ is
the outcome of period $k$. {This definition of test is based on Dawid's
(1984) \emph{prequential principle} that assessment of a theory $\mu$ should
depend on $\mu$ only through the sequence of forecasts that $\mu$ made over
the realized sequence of outcomes, and not on forecasts conditional on data
that has not been realized. {More general tests allow the tester (Alice) to
use forecasts of Bob's theory after any finite history, including the
histories that did not occur. This paper focuses on prequential tests.} }

The test is \emph{finite} if, for some $n\in\mathbb{N}$, $T$ depends only on
the forecasts and outcomes of the first $n$ periods.

\begin{definition}
\label{acceptance-set}Fix a prequential test $T$ and let $\mu$ be a theory.

\begin{enumerate}
\item $\mu$ \emph{passes the test over a realization $x=(s_0,s_1,\dots)\in S^%
\mathbb{N}$} if $\mu(N(s_0,\dots,s_{n-1}))>0$ for every $n$ and $%
T(p_0,s_0,p_1,s_1,\dots)=\text{PASS}$, where $p_k$ is the forecast of $\mu$
over $x$ at period $k$.

\item The \emph{acceptance set of $\mu$ under $T$}, denoted $A_\mu^{(T)}$ is
the set of all realizations $x\in S^\mathbb{N}$ such that $\mu$ passes the
test over $x$.
\end{enumerate}
\end{definition}

\begin{definition}
\label{accept}The prequential test $T$ \emph{accepts the data generating
process in a paradigm $\Gamma$ with probability $1-\epsilon$} if $%
\mu\left(A_\mu^{(T)}\right)\ge 1-\epsilon$ for every $\mu\in \Gamma$.
\end{definition}

If the test accepts the data generating process with high probability and
Bob's theory $\mu$ is the data generating process then Bob is assured that
he will pass the test with high probability.

\begin{definition}
\label{manipulate}Let $\Gamma$ be a paradigm. A test $T$ is $\epsilon$%
-manipulable w.r.t $\Gamma$ if there exists some distribution $\zeta$ over $%
\Gamma$ such that 
\begin{equation*}
\zeta\otimes\mu\left(\left\{(\nu,x)\colon x\in
A_\nu^{(T)}\right\}\right)>1-\epsilon
\end{equation*}
for every theory $\mu\in\Gamma$.
\end{definition}

If a test is $\epsilon$-manipulable then Bob can randomize a theory
according to $\zeta$ and be sure that, with high probability, he passes the
test regardless of the data generating process $\mu$.

If a test is not manipulable then every strategy to fake theories fails on
at least one data-generating process. Following Dekel and Feinberg (2006) and Olszewski and Sandroni (2009b), we use in this paper a stronger
condition, which entails that every strategy employed by an ignorant expert
will fail on a topologically large set of data generating processes. Recall
that a set $M$ of a separable metric space $Z$ is co-meager if it contains a
countable intersection of dense open sets of $Z$. We say that a property is
satisfied by co-meager many elements in $Z$ if the set of elements that
satisfy this property is co-meager.
\begin{definition}
\label{def:strongly} Let $\Gamma$ be a paradigm. A test $T$ is \emph{%
strongly non-manipulable w.r.t $\Gamma$} if, for every distribution $\zeta$
over $\Gamma$, 
\begin{equation}
\zeta\otimes\mu\left(\left\{(\nu,x)\colon x\in
A_\nu^{(T)}\right\}\right)=0
\end{equation}
for co-meager many $\mu$-s in $\Gamma$.
\end{definition}

Note that in Definition~\ref{def:strongly} we require that the set of $\mu$
on which the expert fail be co-meager relative to the topological space $%
\Gamma$. We do not assume that $\Gamma$ itself is a co-meager subset of $%
\Delta(\Omega)$. Indeed, most paradigms used in modeling, i.e., most
interesting classes of stochastic processes (stationary processes, markov
processes, mixing processes) are meager as subsets of $\Delta(\Omega)$.
\section{Results}

\label{results}

\subsection{Previous results}

The following result extends an argument of Sandroni (2003) to convex and
compact paradigms. The proof relies on Fan's Minimax Theorem. See also Vovk,
V. and G. Shafer (2005) for a related result. For completeness, we reproduce
the proof in the appendix. 
\begin{proposition}
\label{finite-old} Fix $\epsilon>0$ and let $\Gamma$ be a convex and compact
paradigm. Let $T$ be a finite test. If $T$ accepts the data generating
process in the paradigm $\Gamma$ with probability $1-\epsilon$ then $T$ is $%
(\epsilon+\delta)$-manipulable w.r.t $\Gamma$ for every $\delta>0$.
\end{proposition}

The following result was proved by Shmaya (2008). The proof relies on
Martin's Theorem about determinacy of Blackwell games (Martin (1998)).
\begin{proposition}
\label{infinite-old}Fix $\epsilon>0$ and let $\Gamma=\Delta(\Omega)$ be the
paradigm of all theories. Let $T$ be a prequential test. If $T$ accepts the
data generating process in the paradigm $\Gamma$ with probability $%
1-\epsilon $ then $T$ is $(\epsilon+\delta)$-manipulable w.r.t $\Gamma$ for
every $\delta>0$.
\end{proposition}
\subsection{New result}
Our new result is that, for the compact and convex paradigm of exchangeable
processes, there exists an infinite prequential test which is
non-manipulable. Therefore, Proposition~\ref{infinite-old} about
manipulability of infinite tests does not hold for convex and compact
paradigms. In contrast, Proposition~\ref{finite-old} about manipulability of
finite tests holds for every compact and convex paradigm.

For the rest of the paper we fix the set of outcome $S=\{0,1\}$. Theorem~\ref%
{thetheorem} below and its proof applies with minor changes to the case of
arbitrary finite set of outcomes, however since we are mainly interested in
a counter-example the restriction to a two-elements outcome set is worth the
notational convenience it provides. We identify the set forecast $p\in
\Delta(S)$ with $p[1]\in [0,1]$.

Before stating our result, we give two simple examples of paradigms for which Proposition~\ref{infinite-old} does not hold. In Example~\ref{ex:compact} the paradigm is compact but not convex and in Example~\ref{ex:convex} the paradigm is convex but not compact. 
\begin{example}\label{ex:compact}Let $\Gamma$ be the paradigm of all i.i.d distributions. Let $T$ be the test that passes the expert if and only if \begin{equation}\label{eq:trivial}p_n=\limsup_{k\rightarrow\infty}(s_0+\dots+s_{k-1})/k\end{equation} 
for every $n\ge 0$ (so that the expert passes if he makes the same prediction in all days and this prediction matches the empirical frequency of outcomes). Then $T$ accepts the data-generating process in $\Gamma$ with probability $1$ and is strongly non-manipulable w.r.t.\ $\Gamma$.\end{example}
\begin{example}\label{ex:convex}Let $\Gamma$ be the convex hull of the paradigm of all i.i.d.\ distributions. Let $T$ be the test that passes the expert if and only if~\eqref{eq:trivial} is satisfied for all but finitely many $n$-s. Then $T$ accepts the data-generating process in $\Gamma$ with probability $1$ and is strongly non-manipulable w.r.t.\ $\Gamma$.\end{example}

A theory $\mu\in\Delta(S^\mathbb{N})$ is \emph{exchangeable} if $%
\mu(N(s_0,\dots,s_{n-1}))=\mu(N(s_{\pi(0)},\dots,s_{\pi(n-1)}))$ for every
partial realization $(s_0,s_1,\dots,s_{n-1})\in S^{<\mathbb{N}}$ and every
permutation $\pi$ over $\{0,1,\dots,n-1\}$. Clearly, the paradigm $\Gamma$
of exchangeable processes is compact and convex. Moreover, it follows from De-Finetti's Theorem (see below) that $\Gamma$ is the closed convex hull of the paradigm of all i.i.d.\ distributions.
\begin{theorem}
\label{thetheorem}Let $\Gamma$ be the paradigm of exchangeable processes.
Then there exists a prequential test $T$ such that $T$ accepts the
data-generating process in $\Gamma$ with probability $1$ and $T$ is strongly
non-manipulable w.r.t $\Gamma$.
\end{theorem}
\subsection{Additional comments}
\subsubsection*{Non-prequential tests}Proposition~\ref{infinite-old} relies on Dawid's prequential principle. If
one abandons this principle, allowing Alice to use forecasts of Bob's theory
after all finite histories, including those that did not occur, then some non-manipulable tests do
exist. In particular, Olszewski and Sandroni (2009b) construct a strongly non-manipulable
non-prequential test for the paradigm of all theories. Shmaya (2008) (Section 7) also have non-prequential
non-manipulable tests for the paradigm of all theories. In addition, Al-Najjar et al. (2010) construct a (non-prequential) non-manipulable test for the paradigm of theories which are learnable and
sufficient for predictions, which is convex but not compact.
\subsubsection*{Non Borel tests}Another assumption of Proposition~\ref{infinite-old} is that the test is a Borel function (This assumption is part of our definition of test in Section~\ref{setup}). Shmaya (2008) shows that under the axiom of choice, there exist non-manipulable prequential  tests which are not Borel.
\subsubsection*{Future independent tests}Olszewski and Sandroni (2008) introduce the class of future independent tests. For prequential tests, the future independence property can be seen as weakening of the finite horizon property: instead of requiring that the test terminates at finite time, it only requires that if the expert fail the test then this failure will be demonstrated in finite (possibly unbounded) time. 
Olszewski and Sandroni (2009a) show that many of their manipulability results
for future independent tests can also be proved for convex and
compact paradigms. The test we construct is not future independent. 
\subsubsection*{Strong non-manipulability}The definition of strong non-manipulability appeals to a topological notion
of `large' set of distribution as a co-meager set. The goal of the
definition is to capture the idea that, regardless of the strategy he uses
to creates forecasts, the uninformed expert will fail on a large set of
stochastic processes. This definition has been used in the expert testing
literature, but we recognize that such topological notion has
well-known drawbacks. In particular, it does not indicate odds of
discrediting the uninformed expert.  We view the main contribution of our
theorem to be the very fact that there exists a test which is not
manipulable. The strong non-manipulability property is a bonus.
 \section{Proof of Theorem~\protect\ref{thetheorem}}
\label{proof} 

\subsection{Preliminaries}

\subsubsection*{De-Finetti's Theorem}

For every $\lambda\in \Delta([0,1])$ let $\varepsilon_\lambda\in
\Delta(\{0,1\}^\mathbb{N})$ be the distribution of infinite sequence of
i.i.d coins with probability $q$ of success, where $q$ is drawn from $\lambda
$: 
\begin{equation}  \label{mu-lambda}
\varepsilon_\lambda(N(s_0,\dots,s_{n-1}))=\int%
\prod_{i=0}^{n-1}q^{s_i}(1-q)^{1-s_i}\lambda(\text{d} q).
\end{equation}
Clearly the distribution $\varepsilon_\lambda$ is exchangeable. De-Finetti's
Theorem states that the map $\lambda\in\Delta([0,1])\mapsto\varepsilon_%
\lambda$ is one-to-one and onto the set $\Gamma$ of exchangeable
distributions over $\{0,1\}^\mathbb{N}$, and its inverse is given by $%
\mu\in\Gamma\mapsto{\bar\mu}\in\Delta([0,1])$ where ${\bar\mu}\in
\Delta([0,1])$ is the push-forward of $\mu$ under $L$ (i.e., ${\bar\mu}%
(B)=\mu(L^{-1}(B))$ for every Borel subset $B$ of $[0,1]$) and $L:\{0,1\}^%
\mathbb{N}\rightarrow [0,1]$ is the limit average 
\begin{equation}  \label{limit-average}
L(s_0,s_1,\dots)= \limsup_{n\rightarrow\infty}(s_0+\dots+s_{n-1})/n.
\end{equation}
% De-Finetti
% celebrated theorem states that every exchangeable distribution $%
% \mu\in\Delta(\{0,1\}^\mathbb{N})$ is of the form $\varepsilon_{\bar\mu}$ for
% a unique ${\bar\mu}\in \Delta([0,1])$. The probability measure ${\bar\mu}%
% \in\Delta([0,1])$ that corresponds to an exchangeable $\mu\in\Delta(\{0,1\}^%
% \mathbb{N})$ is given by the distribution of the limit averages of a $\mu$%
% -random realization. Formally, let $L:\{0,1\}^\mathbb{N}\rightarrow [0,1]$
% be the limit average 
% \begin{equation}  \label{limit-average}
% L(s_0,s_1,\dots)= \limsup_{n\rightarrow\infty}(s_0+\dots+s_{n-1})/n.
% \end{equation}
% For an exchangeable distribution $\mu\in \Delta(\{0,1\}^\mathbb{N})$, let ${%
% \bar\mu}\in \Delta([0,1])$ be the push-forward of $\mu$ under $L$, i.e. ${%
% \bar\mu}(B)=\mu(L^{-1}(B))$ for every Borel subset $B$ of $[0,1]$.
% \begin{definetti-theorem}
% The map $\lambda\in\Delta([0,1])\mapsto\varepsilon_\lambda$ is one-to-one
% and onto the set $\Gamma$ of exchangeable distributions over $\{0,1\}^%
% \mathbb{N}$ and its inverse is given by $\mu\in\Gamma\mapsto{\bar\mu}%
% \in\Delta([0,1])$:% \begin{enumerate}
% % \item Let $\mu\in\Delta(\{0,1\}^\bbn)$ be exchangeable. Then $\mu=\varepsilon_\mubar$.
% % \item Let $\lambda\in \Delta([0,1])$ and let $\mu=\varepsilon_\lambda$. Then $\bar\mu=\lambda$.\end{enumerate}
% \end{definetti-theorem}
Since the map $\lambda\mapsto\varepsilon_\lambda$ is continuous and its
domain $\Delta([0,1])$ is compact it follows that the maps $%
\lambda\mapsto\varepsilon_\lambda$ and $\mu\mapsto\bar\mu$ are
homeomorphisms.

\subsubsection*{A non-manipulable test}

Our proof uses the following proposition. % \begin{proposition}
% \label{olszewski-sandroni-interval}There exists a Borel function $%
% t:\Delta([0,1])\otimes [0,1]\rightarrow \{\text{FAIL},\text{PASS}\}$ such
% that
% \begin{enumerate}
% \item For every $\bar\mu\in\Delta([0,1])$, $\bar\mu\left(\{p|t(\bar\mu,p)=%
% \text{PASS}\}\right)=1$.
% \item For every $\bar\zeta\in\Delta\left(\Delta([0,1])\right)$, one has 
% \begin{equation*}
% \bar\zeta\otimes\bar\mu\left(\left\{(\bar\nu,p)\colon  t(\bar\nu,p)=\text{%
% PASS}\right\}\right)=0
% \end{equation*}
% for co-meager many $\bar\mu\in\Delta([0,1])$.
% \end{enumerate}
% \end{proposition}

\begin{proposition}
\label{olszewski-sandroni-interval}Let $Q$ be a compact metric space and let 
$\Delta(Q)$ be equipped with the weak-$^\ast$ topology. There exists a Borel
function $t:\Delta(Q)\times Q\rightarrow \{\text{FAIL},\text{PASS}\}$ such
that

\begin{enumerate}
\item For every $\bar\mu\in\Delta(Q)$, $\bar\mu\left(\{q\in Q|t(\bar\mu,q)=%
\text{PASS}\}\right)=1$.

\item For every $\bar\zeta\in\Delta\left(\Delta(Q)\right)$, one has 
\begin{equation*}
\bar\zeta\otimes\bar\mu\left(\left\{(\bar\nu,q)\in\Delta(Q)\times Q\colon 
t(\bar\nu,q)=\text{PASS}\right\}\right)=0
\end{equation*}
for co-meager many $\bar\mu\in\Delta(Q)$.
\end{enumerate}
\end{proposition}

This proposition has the following nice interpretation: Assume that an
expert claims to know the distribution from which an element $q\in Q$ is
drawn. Say that the expert passes the test if $t(\bar\mu,q)=\text{PASS}$,
where $\bar\mu$ is the distribution proclaimed by the expert and $q$ is the
actual realization. The first property says that if the expert provides the
correct distribution of $q$ then he passes with probability $1$. The second
property ensures that, if the expert randomizes his theory according to some
distribution $\bar\zeta$, then he will fail the test with probability $1$
for a large (co-meager) set of truths.

Proposition~\ref{olszewski-sandroni-interval} was proved by Olszewski and
Sandroni (2009b). Their theorem is formulated for the underlying space of
realization $Q=\{0,1\}^\mathbb{N}$ but their theorem and proof carry
through for every compact metric space. We need the case $Q=[0,1]$.

\subsection{Proof}

We first prove that when forecasts are made using an exchangeable theory,
then the forecasts over a single realization already determines the entire
theory.

\begin{lemma}
Let $x=(s_0,s_1,\dots)\in\{0,1\}^\mathbb{N}$ be a realization and let $%
\mu^{\prime }=\varepsilon_{\lambda^{\prime }}$ and $\mu^{\prime \prime
}=\varepsilon_{\lambda^{\prime \prime }}$ be two exchangeable theories. If $%
\mu^{\prime }(1|s_0,\dots,s_{n-1})=\mu^{\prime \prime }(1|s_0,\dots,s_{n-1})$
for every $n\in\mathbb{N}$ then $\lambda^{\prime }=\lambda^{\prime \prime }$
(and, in particular, $\mu^{\prime }=\mu^{\prime \prime }$).
\end{lemma}

\begin{proof}
Under the assumption of the lemma, it follows by induction from~(\ref{p-n-s}%
) that 
\begin{equation}  \label{mu-mu-prime}
\mu^{\prime }\left(N\left(s_0,\dots,s_{n-1}\right)\right)=\mu^{\prime \prime
}\left(N\left(s_0,\dots,s_{n-1}\right)\right)
\end{equation}
for every $n\in \mathbb{N}$. % By~(\ref{mu-lambda}) we get that 
% \begin{equation*}
% \begin{split}
% \mu^{\prime
% }&\left(N\left(s_0,\dots,s_{n-1}\right)\right)=\int%
% \prod_{i=0}^{n-1}p^{s_i}(1-p)^{1-s_i}\lambda^{\prime }(\text{d} p),\text{ and%
% } \\
% \mu^{\prime \prime
% }&\left(N\left(s_0,\dots,s_{n-1}\right)\right)=\int%
% \prod_{i=0}^{n-1}p^{s_i}(1-p)^{1-s_i}\lambda^{\prime \prime }(\text{d} p)
% \end{split}%
% \end{equation*}
% Therefore it follows from~(\ref{mu-mu-prime}) that 
From~(\ref{mu-mu-prime}) and~(\ref{mu-lambda}) it follows that 
\begin{equation*}
\int\prod_{i=0}^{n-1}q^{s_i}(1-q)^{1-s_i}\lambda^{\prime }(\text{d}
q)=\int\prod_{i=0}^{n-1}q^{s_i}(1-q)^{1-s_i}\lambda^{\prime \prime }(\text{d}
q)
\end{equation*}
for every $n$. Since $\prod_{i=0}^{n-1}q^{s_i}(1-q)^{1-s_i}$ is a polynomial
in $q$ of degree $n$, it follows again by induction that 
\begin{equation*}
\int q^n\lambda^{\prime }(\text{d} q)=\int q^n\lambda^{\prime \prime }(\text{%
d} q)
\end{equation*}%
for every $n$. Since a probability measure with bounded support is
determined by its moments it follows that $\lambda^{\prime }=\lambda^{\prime
\prime }$.
\end{proof}

\begin{corollary}
\label{corollary-moments}There exists a Borel subset $B$ of $[0,1]^\mathbb{N}%
\times\{0,1\}^\mathbb{N}$ and a Borel map $f:B\rightarrow \Delta([0,1])$
such that for every exchangeable theory $\mu$ and for every realization $%
x=(s_0,s_1,\dots)$ one has 
\begin{equation*}
\begin{split}
&(p_0,p_1,\dots,s_0,s_1,\dots)\in B,\text{ and} \\
&f(p_0,p_1,\dots,s_0,s_1,\dots)=\bar\mu.
\end{split}%
\end{equation*}
where $p_n=\mu(1|s_0,\dots,s_{n-1})$ and $\bar\mu$ is the push-forward of $%
\mu$ under the limit average $L:\{0,1\}^\mathbb{N}\rightarrow [0,1]$ defined
in~(\ref{limit-average}).
\end{corollary}

\begin{proof}
Consider the Borel map $g:\Delta([0,1])\times\{0,1\}^\mathbb{N}\rightarrow
[0,1]^\mathbb{N}\times\{0,1\}^\mathbb{N}$ that is given by $%
g(\lambda,s_0,s_1,\dots)=(p_0,p_1,\dots,s_0,s_1,\dots)$ where $%
p_n=\mu(1|s_0,\dots,s_{n-1})$ for $\mu=\varepsilon_\lambda$. By the previous
lemma this map is one-to-one. Therefore its image $B=g\left(\Delta([0,1])%
\times\{0,1\}^\mathbb{N}\right)$ is a Borel set and the inverse map $%
g^{-1}:B\rightarrow \Delta([0,1])\times\{0,1\}^\mathbb{N}$ is a Borel map
(See Kechris (1994), Corollary 15.2). Take $f=\pi\circ g^{-1}$ where $%
\pi:\Delta([0,1])\times\{0,1\}^\mathbb{N} \rightarrow\Delta([0,1])$ is the
projection on the first coordinate. The assertion in the corollary follows
from the fact that if $\mu$ is exchangeable then $\mu=\varepsilon_\lambda$
when $\lambda=\bar\mu$.
\end{proof}

\begin{proof}[Proof of Theorem~\protect\ref{thetheorem}]
Consider the test $T:[0,1]^\mathbb{N}\times\{0,1\}^\mathbb{N}%
\rightarrow\{0,1\}$ which is given by 
\begin{equation*}
T(p,x)=%
\begin{cases}
t\bigl(f(p,x), L(x)\bigr), & \text{ if }(p,x)\in B \\ 
\text{FAIL}, & \text{ otherwise.}%
\end{cases}%
\end{equation*}
for every realization $x=(s_0,s_1,\dots)\in\{0,1\}^\mathbb{N}$ and every
sequence $p=(p_0,p_1,\dots)\in [0,1]^\mathbb{N}$ of forecasts, where $%
B\subseteq [0,1]^\mathbb{N}\times\{0,1\}^\mathbb{N}$ and $f:B\rightarrow
\Delta([0,1])$ are as in Corollary~\ref{corollary-moments}, $L:\{0,1\}^%
\mathbb{N}\rightarrow [0,1]$ is the limit average~(\ref{limit-average}), and 
$t$ is the test defined in Proposition~\ref{olszewski-sandroni-interval} for
the interval $Q=[0,1]$. From Corollary~\ref{corollary-moments} and the
definition of $T$ it follows that 
\begin{equation}  \label{T-t}
\mu\text{ passes }T\text{ over }x\Longleftrightarrow t(\bar\mu, L(x))=\text{%
PASS}
\end{equation}
for every realization $x\in\{0,1\}^\mathbb{N}$ and every exchangeable
distribution $\mu$. We now claim that the desired properties of $T$ follow
from the corresponding properties of $t$. Indeed, for every exchangeable $%
\mu $ 
\begin{equation*}
\mu\left(A_\mu^{(T)}\right)=\mu\left(\{x|t(\bar\mu,L(x))=\text{PASS}%
\}\right)=\bar\mu\left(\{q|t(\bar\mu,q)=\text{PASS}\}\right)=1
\end{equation*}
where the first equality follows from~(\ref{T-t}), the second from the
definition of $\bar\mu$ and the third from Proposition~\ref%
{olszewski-sandroni-interval}. Thus $T$ accepts exchangeable data generating
processes with probability $1$.

Let $\zeta\in\Delta(\Gamma)$ and let $\bar\zeta\in\Delta\left(\Delta([0,1])%
\right)$ be the push-forward of $\zeta$ under the map $\nu\mapsto\bar\nu$ 
% , i.e. 
%\begin{equation}  \label{barzeta}
%\bar\zeta(B)=\zeta\left(\{\nu\in\Gamma|\bar\nu\in B\}\right)
%\end{equation}
%for every Borel subset $B$ of $\Delta([0,1])$, 
where $\bar\nu\in\Delta([0,1])$ is the pushforward of $\nu$ under $L$ for
every $\nu\in\Gamma$. Fix a theory $\mu\in\Gamma$ and let $\bar\mu$ be the
pushforward of $\mu$ under $L$. Then%, for every exchangeable $\mu$,
\begin{multline*}
\zeta\otimes\mu\left(\left\{(\nu,x)\colon x\in
A_\nu^{(T)}\right\}\right)= \\
\zeta\otimes\mu\left(\left\{(\nu,x)\colon t(\bar\nu,L(x))=\text{PASS}%
\right\}\right)= \bar\zeta\otimes\bar\mu\left(\left\{(\bar\nu,q)\left
|t(\bar\nu,q)=\text{PASS}\right.\right\}\right)
\end{multline*}
where the first equality follows from~(\ref{T-t}) and the second from the
definitions of $\bar\zeta$ and $\bar\mu$. Therefore the set of all $\mu$-s
such that $\zeta\otimes\mu\left(\left\{(\nu,x)\colon T(\nu,x)=\text{PASS}%
\right\}\right)=0$ is the preimage of the set of all $\bar\mu$ such
that $\bar\zeta\otimes\bar\mu\left(\left\{(\bar\nu,q)\colon t(\bar\nu,q)=%
\text{PASS}\right\}\right)=0$ under the map $\mu\mapsto\bar\mu$.
Since the later set is co-meager by Theorem~\ref{olszewski-sandroni-interval}
and the map $\mu\mapsto\bar\mu$ is homeomorphism, it follows that the former
set is co-meager, as desired.
\end{proof}

\section{Testing and merging}

\label{merging} The expert testing literature asks whether an uninformed
expert can produce forecasts that appear to be are as good as the forecasts
of the informed given the observed outcomes. A related question is whether
an uninformed expert can produce forecasts that are close to the forecasts
of the informed expert. In this section we show that the answer to these two
questions might be different.

Recall that a distribution $\nu\in\Delta(S^\mathbb{N})$ \emph{merges} with a
distribution $\mu\in\Delta(S^\mathbb{N})$ if 
\begin{equation*}
\nu\left(s|s_0,\dots,s_{n-1}\right)-\mu\left(s|s_0,\dots,s_{n-1}\right)%
\xrightarrow[n\rightarrow\infty]{}0
\end{equation*}
for $\mu$-almost every realization $(s_0,s_1,s_2,\dots)$. This notion of
merging was introduced by Kalai and Lehrer (1994) as a weakening of
Blackwell and Dubins merging (1962). Sorin (1999) relates merging with
several game theoretic models. Kalai et al.\ (1999) relate merging with certain classes of tests.

Say that a paradigm $\Gamma$ is \emph{learnable} if there exists $%
\nu\in\Gamma$ such that $\nu$ merges with every $\mu\in\Gamma$. If $\Gamma$
is learnable then an uninformed expert can learn to predict, i.e., to provide
forecasts that becomes arbitrary close to the forecasts of the informed
expert.

The paradigm of all exchangeable theories is learnable: If $\Omega=\{0,1\}$
then the theory $\nu=\varepsilon_\lambda$ given in~\eqref{mu-lambda} where $%
\lambda$ is the uniform distribution on $[0,1]$ merges with all exchangeable
theories. However, by our Theorem~\ref{thetheorem}, there exists a test that
screens an uninformed expert from the informed expert based on the outcomes
of the process.

The paradigm of all theories is not learnable, so that an uninformed expert
cannot produce forecasts that are close to the informed expert's forecasts
for every process, but, by Proposition~\ref{infinite-old}, he can do as good
as the informed expert in every prequential test. \appendix
\section{Proof of Proposition~\protect\ref{finite-old}}

\label{proof-old} Let $F:\Delta(\Gamma)\otimes\Gamma\rightarrow [0,1]$ be
given by $F(\zeta,\mu)=\zeta\otimes\mu\left(\left\{(\nu,x)\in\Gamma\times%
\Omega\colon x\in A_\nu^{(T)}\right\}\right)$. Then $F$ is bilinear
and, since $T$ is finite, continuous in its second argument. It follows from
Fan's Minimax Theorem (1953) that 
\begin{equation}  \label{fan}
\sup_\zeta\min_\mu F(\zeta,\mu)=\min_\mu\sup_\zeta F(\zeta,\mu),
\end{equation}
where the suprema are over all $\zeta\in\Delta(\Gamma)$ and the minima are
over all $\mu\in\Gamma$.

For every $\mu\in\Gamma$, if $\zeta$ is Dirac atomic measure over $\mu$ then 
$F(\zeta,\mu)\ge 1-\epsilon$ since $T$ accepts data generating process in $%
\Gamma$ with probability $1-\epsilon$. Therefore it follows from~(\ref{fan})
that for every $\delta>0$ there exists $\zeta\in\Delta(\Gamma)$ such that $%
\min_\mu F(\zeta,\mu)>1-(\epsilon+\delta)$, i.e., that $T$ is $%
(\epsilon+\delta)$-manipulable.

\end{document}